\documentclass [10pt,a4paper]{amsart}
\usepackage {hannah}
\usepackage {color,graphicx,psfrag,enumitem}
\usepackage{texdraw}
\usepackage [latin1]{inputenc}
\usepackage {amsmath,amssymb}
\usepackage{amsfonts}
\usepackage{amsthm}
\usepackage{graphicx,color}
\graphicspath{{pics/}}
\usepackage [latin1]{inputenc}
\usepackage {times}
\usepackage {float}
\usepackage{hyperref}
\hypersetup{colorlinks=true}
\usepackage{amscd}
\usepackage[shortalphabetic,nobysame]{amsrefs}
\usepackage{todonotes}

\newcommand{\AAA}{\mathfrak A}
\newcommand{\BBB}{\mathfrak B}

\newcommand{\CCC}{\mathcal C}
\newcommand{\DDD}{\mathcal D}

\newcommand{\T}{{\mathds T}}

\DeclareMathOperator {\val}{val}
\DeclareMathOperator {\trop}{trop}
\DeclareMathOperator {\Trop}{Trop}
\DeclareMathOperator {\refi}{ref}
\DeclareMathOperator {\conj}{conj}
\DeclareMathOperator {\VV}{Vert}
\DeclareMathOperator {\EE}{Edge}
\DeclareMathOperator {\ma}{mark}
\newcommand {\fix}{\text{fix}}

\title {Tropical real Hurwitz numbers}
\author{Hannah Markwig and Johannes Rau}
\address{Universit\"at des Saarlandes, Fachrichtung
  Mathematik, Postfach 151150, 66041 Saarbr\"ucken, Germany }
\email{hannah@math.uni-sb.de, johannes.rau@math.uni-sb.de}

\thanks{\emph{2010 Mathematics Subject Classification:} Primary 14T05, Secondary 14N10.}
\keywords {Tropical geometry, Hurwitz numbers, enumerative geometry}

\begin{document}

\begin{abstract}
In this paper, we define tropical analogues of real Hurwitz numbers, 
i.e.\ numbers of covers of surfaces with compatible involutions satisfying prescribed ramification properties. 
We prove a correspondence theorem stating the equality of the tropical numbers with their real counterparts.
We apply this theorem to the case of double Hurwitz numbers 
(which generalizes our result from \cite{GPMR14}). 
\end{abstract}

\maketitle

\section{Introduction}

\label{sec-intro}
\subsection{Result}
We study tropicalizations of covers of surfaces with compatible orientation-reversing involutions and define tropical analogues of real Hurwitz numbers, i.e.\ numbers of such covers satisfying fixed ramification data. We prove a correspondence theorem stating the equality of tropical real Hurwitz numbers and real Hurwitz numbers. 
As in the complex case (i.e.\ without the extra data of real involutions), 
the basic idea is that tropical covers can be considered as graphical representations of
pair-of-pants decompositions or of the terms in the well-known degeneration formula for Hurwitz numbers 
(see e.g.\ \cites{BBM10, Cav10}). 

\subsection{Motivation and background}
The definition of Hurwitz numbers as numbers of ramified covers of a fixed target curve satisfying prescribed conditions goes back to Hurwitz and has since then provided fruitful connections between various areas of mathematics such as algebraic geometry, representation theory, the theory of random matrices and mathematical physics (see e.g.\ \cites{elsv:hnaiomsoc, OP, OP06}). Real versions of Hurwitz numbers have been considered, e.g.\ in \cite{Cad05}. 
They also appear in the study of topological field theories \cite{AS06}. At first glance, they lack an important feature that complex Hurwitz numbers have: they are not invariants of the position of the chosen branch points. In light of the success of tropical methods for the study of real analogues of numbers of plane curves (which is summarized in the next subsection), we believe that the study of tropical methods for real Hurwitz numbers that we initiated in \cite{GPMR14} with Guay-Paquet and continue here will be fruitful for the further research in this area.

\subsubsection{Tropical Hurwitz numbers}
Tropical versions of (complex) double Hurwitz numbers of $\P^1$ have been studied in \cite{CJM10} using an approach via the symmetric group. A more general correspondence theorem obtained with topological methods was proved in \cite{BBM10}. Tropical Hurwitz numbers appear in the tropical computation of genus zero Zeuthen numbers, i.e.\ numbers of curves satisfying point conditions and tangency conditions to fixed lines \cite{BBM11}. Tropical double Hurwitz numbers are useful to prove statements about the structure of double Hurwitz numbers \cite{cjm:wcfdhn}. 

In \cite{GPMR14}, together with Guay-Paquet we define 
tropical real double Hurwitz numbers with positive branch points
and study their combinatorial properties.

In this paper, we generalize the correspondence theorem to arbitrary covers.
Our approach is similar to the open Hurwitz numbers approach in \cite{BBM10}.

\subsubsection{Real enumerative geometry and tropical geometry}
Tropical geometry is a powerful tool for the study of enumerative problems \cite{Mi03}. It has particular success in real enumerative geometry and the study of Welschinger invariants \cites{IKS06, Shu06b}. The latter can be viewed as analogues of plane Gromov-Witten invariants, i.e.\ of numbers of (complex) nodal plane curves of some fixed degree and genus satisfying point conditions. More precisely, when counting complex curves the number we obtain does not depend on the particular choice of the point conditions, as long as they are in general position. For real curves this is no longer true. However, when we count each curve with a suitable sign depending on the nodes, we obtain an invariant count, the Welschinger number \cite{Wel05}. Obviously, it is a lower bound for the number of real curves. The tropical approach provides algorithms to determine Welschinger invariants. It can also be used for interesting statements on the asymptotics when compared to Gromov-Witten invariants (e.g.\ \cite{IKS05}).

The situation is very similar for real Hurwitz numbers. As mentioned above, these numbers do depend
on the position of the branch points (cf.\ \ref{ex-realbranchpoints}). Hence the question
arises whether a suitable sign rule exists to make the count invariant. 
In this context we would like to mention work of Itenberg and Zvonkine (not yet published) in which the authors
define such a signed count of real polynomials (i.e.\ real Hurwitz numbers of genus $0$ and with a total ramification point)
and show the invariance of the position of the branch points. 
The questions how these signed counts relate to the tropical picture and whether such signed counts
also exist for more general Hurwitz numbers is subject of current research.

\subsection{Organization of this paper}
In section \ref{sec-realcovers}, we introduce covers of surfaces and numbers of such covers satisfying fixed ramification data. We concentrate on the case 
of real Hurwitz numbers, i.e.\ numbers of covers of surfaces with compatible 
orientation-reversing involutions. In section \ref{sec-tropicalcurves}, we introduce tropical curves, covers and real structures
which turn out to be natural counterparts of real covers. We also introduce tropical real Hurwitz numbers. In section \ref{sec-corres}, we prove our correspondence theorem
stating the equality of real Hurwitz numbers with their tropical counterparts. In section \ref{sec-realbranch}, we focus on the case of covers of $\P^1$
with only real branch points. In this case, we can boil down our general definition of real tropical Hurwitz numbers to a more combinatorial recipe. We also recover
the correspondence theorem of \cite{GPMR14} dealing with double Hurwitz numbers with positive real simple branch points.

\subsection{Acknowledgements}
The first author is supported by DFG-grant MA 4797/1-2 and GIF grant no. 1174-197.6/2011. 
The second author would like to thank the Universit\'e de Gen\`eve for the hospitality during his stay.
We would like to thank Erwan Brugall\'e, Ilia Itenberg and Grisha Mikhalkin for helpful discussions. We also thank Maxim Karev and an anonymous referee for helpful remarks on  earlier versions of this paper.

\section{Real covers and Hurwitz numbers}\label{sec-realcovers}
We start by recalling the general definition of (complex) Hurwitz numbers.
Fix a genus $g$, a degree $d$ and a collection $\mu=(\mu_1,\ldots,\mu_n)$ of $n$ partitions of $d$.
Let $\mathcal D$ be a connected oriented closed compact topological surface of genus $h$
and fix $n$ points $p_1, \ldots, p_n \in \mathcal D$.
We want to count ramified covers of degree $d$, i.e.\ maps $f:\mathcal C\rightarrow \mathcal D$ 
where $\mathcal C$ is a connected orientable closed compact topological surface of genus $g$ and 
$f$ is a continuous map which restricts to a degree $d$ covering over $\mathcal D \setminus \{p_1, \ldots, p_n\}$ and which has ramification profile $\mu_i$ over $p_i$ (i.e.\ the multiset of
ramification indices of the preimages of $p_i$ equals $\mu_i$). 
Given two such covers $f:\mathcal C\rightarrow \mathcal D$ and $f':\mathcal C' \rightarrow \mathcal D$,
an isomorphism of covers is a homeomorphism $\varphi : \mathcal C \rightarrow \mathcal C'$ such that 
$f = f' \circ \varphi$. 
Moreover, for our purposes it will sometimes be useful to add markings to the ramification points. 
More precisely, a \emph{marking} of a cover $f$ is a choice of labellings
$q_{i,1}, \ldots, q_{i,l(\mu_i)}$ for the preimage $\{q_{i,1}, \ldots, q_{i,l(\mu_i)}\} = f^{-1}(p_i)$ of each branch point such that the ramification index
at $q_{i,j}$ is $\mu_{i,j}$ (we consider $\mu_i$ as a vector here). 
We require that an isomorphism of marked covers respects the labels, i.e.\ $\varphi(q_{i,j}) = q'_{i,j}$.

\begin{definition}\label{def:complexhn}
  We define the \emph{Hurwitz number} $H^{\C}_{g}(\mathcal D,\mu)$ to be the weighted number of 
	isomorphism classes of covers $f:\mathcal C\rightarrow \mathcal D$. 
	Each cover is weighted by $\frac{1}{|\Aut(f)|}$.
	Analogously, we let $H^{\ma, \C}_{g}(\mathcal D,\mu)$ be the number of marked covers
  (with respect to marked iso-/automorphisms).
\end{definition}

\begin{remark}
$H^{\C}_{g}(\mathcal D,\mu)$ is non-zero only if the prescribed data $g$, $h$ and $\mu$
satisfies the Riemann-Hurwitz formula. We will assume this in the following. \\
Let 
\[
  \Aut(\mu) = \prod_{i=1}^n \{\sigma \in S_{l(\mu_i)} : 
	  \mu_{i,\sigma(j)} = \mu_{i,j} \; \forall \; j \}
\]
be the group of automorphisms of $\mu$.
Let $f^{\ma}$ be a marked cover and let $f^\text{unm}$ be the corresponding unmarked
cover. The action of $\Aut(f^\text{unm})$ on the labels of $f^{\ma}$ gives a group
homomorphism
\[
  \kappa : \Aut(f^\text{unm}) \to \Aut(\mu).
\]
We have $\ker(\kappa) = \Aut(f^{\ma})$ and $[\Aut(\mu) : \text{Im}(\kappa)] = m$, where 
$m$ is the number of markings for $f^\text{unm}$ (up to isomorphism). 
Hence $|\Aut(f^\text{unm})| \cdot m = |\Aut(f^{\ma})| \cdot |\Aut(\mu)|$ and
\[
  H^{\ma, \C}_{g}(\mathcal D,\mu) = |\Aut(\mu)| \cdot H^{\C}_{g}(\mathcal D,\mu).
\]
Note that $H^{\C}_{g}(\mathcal D,\mu)$ does not depend on the position of the $n$ branch points, 
nor on the concrete topological surface $\mathcal D$
(as any other choice $(\mathcal D', p'_1, \ldots, p'_n)$ is homeomorphic 
to the given one). \\
Given a complex structure on $\mathcal D$, for each $f : \mathcal C \rightarrow \mathcal D$ there exists a unique complex structure on $\mathcal C$ such that $F$ is a holomorphic ramified cover. 
Hence $H^{\C}_{g}(\mathcal D,\mu)$ can also be interpreted as the count of (non-constant) holomorphic maps
of compact Riemann surfaces with fixed ramification profile. However, we stick to the topological viewpoint
as this simplifies the gluing construction later.
\end{remark}

We now add the data of an orientation-reversing involution $\iota_{\mathcal D}:\mathcal D\rightarrow \mathcal D$, i.e.\ $\iota_{\mathcal D} \circ \iota_{\mathcal D}=\id$. The fixed point set of the involution is called the \emph{real locus} of $\mathcal D$. We consider \emph{real covers}, i.e.\ 
triples $(\mathcal C, f, \iota_{\mathcal C})$ such that $f:\mathcal C \rightarrow \mathcal D$ is a cover as before, $\iota_{\mathcal C}$ is an orientation-reversing involution on $\mathcal C$ and the compatibility condition
$\iota_{\mathcal D}\circ f= f \circ \iota_{\mathcal C}$ holds. 
For such covers to exist, we need that $\mathcal P=\{p_1,\ldots,p_n\}$ is invariant under $\iota_{\mathcal D}$
(not necessarily pointwise), which we assume from now on.
An isomorphism of real covers is required to satisfy 
$\iota_{\mathcal C'}\circ \varphi= \varphi \circ \iota_{\mathcal C}$.

\begin{definition}\label{def-realhn}
 The \emph{real Hurwitz number} $H^{\R}_{g}((\mathcal D,\iota_{\mathcal D}),\mathcal P, \mu)$ 
 is defined to be the  weighted number of real covers 
 $(\mathcal C, f, \iota_{\mathcal C})$ up to isomorphism. 
 Each cover is weighted by $\frac{1}{|\Aut(f)|}$.
	Analogously we define $H^{\ma, \R}_{g}((\mathcal D,\iota_{\mathcal D}),\mathcal P, \mu)$ 
	to the number	of marked covers (with respect to marked real iso-/automorphisms).
\end{definition}

As before, we have
\[
  H^{\ma, \R}_{g}((\mathcal D,\iota_{\mathcal D}),\mathcal P, \mu) =
	  |\Aut(\mu)| \cdot H^{\R}_{g}((\mathcal D,\iota_{\mathcal D}),\mathcal P, \mu).
\]

\begin{remark}\label{rem:independence}
The real Hurwitz number $H^{\R}_{g}((\mathcal D,\iota_{\mathcal D}),\mathcal P, \mu)$ is invariant under homeomorphism respecting the involutions and the branch points.
More precisely, $$H^{\R}_{g}((\mathcal D,\iota_{\mathcal D}),\mathcal P, \mu)= H^{\R}_{g}((\mathcal D',\iota_{\mathcal D'}),\mathcal P', \mu)$$ (where $\mathcal P'=(p_1',\ldots,p_n')$) if there is a homeomorphism
$h:\mathcal D\rightarrow \mathcal D'$ satisfying $h\circ \iota_{\mathcal D}= \iota_{\mathcal D'} \circ h$ and $h(p_i)=p_i'$.
\end{remark}

  Instead of ramified covers of closed surfaces, we can equivalently consider unramified coverings of surfaces
	with boundary. The transition is made by removing open discs around each branch point as well as its preimages
	resp.\ gluing in discs (with a marked point) to the boundary circles. The ramification index of a point
	translates to the degree of the map between the corresponding boundary circles. 
	This framework is referred to as \emph{open} Hurwitz numbers in \cite{BBM10} and is more appropriate to 
	describe the gluing constructions needed in the following. We will therefore adopt this viewpoint from now on.
	Moreover, we need to define the following refined real Hurwitz numbers for which we also prescribe the
	behaviour of the involution at the ramification points. Let us be more precise.

	Consider the (unramified) cover $f : \mathcal C \rightarrow \mathcal D$ of surfaces with boundary
	and equipped with orientation-reversing involutions such that 
	$f\circ \iota_{\mathcal C}= \iota_{\mathcal D} \circ f$.
	Let $\mathcal B = \{B_1,\ldots,B_n\}$ be the collection of boundary circles of $\mathcal D$ and
	$B_{i,1}, \ldots, B_{i,l(\mu_i)}$ be the boundary circles of $\mathcal C$ mapped to $B_i$ (hence
	we fix a marking of the cover). By slight abuse of notation, we denote by $\iota$ the map between
	indices such that $\iota_{\mathcal D}(B_i) = B_{\iota(i)}$ and 
	$\iota_{\mathcal C}(B_{i,j}) = B_{\iota(i,j)}$. 
	
	Note that on any boundary circle which is invariant under $\iota$, the involution has 
	exactly two fixed points. 
	If $B_{i,j}$ is a boundary circle which is invariant under $\iota_{\mathcal C}$ 
	(i.e.\ $\iota(i,j) = (i,j)$) then $B_i$ is $\iota_{\mathcal D}$-invariant and 
	the preimage of a fixed point of $B_i$ in $B_{i,j}$ is a $\iota_{\mathcal C}$-invariant set
	of cardinality $\mu_{i,j}$. 
	Hence the number of $\iota_{\mathcal C}$-fixed points in this set is $1$ if 
	$\mu_{i,j}$ is odd and $0$ or $2$ if $\mu_{i,j}$ is even. 
	In particular, if $\mu_{i,j}$ is even, both fixed points of $B_{i,j}$ are mapped to the same
	fixed point in $B_i$. This gives rise to a map
	\[
	  F : \{(i,j) \in \Fix(\iota): \mu_{i,j} \equiv 0 \; \mod 2\} \to \Fix(\iota_{\mathcal D}) \cap (B_1 \cup \ldots \cup B_n).
	\]
	It assigns to each $\iota$-invariant boundary circle $B_{i,j}$ with $\mu_{i,j}$ even the unique
	fixed point in $B_i$ to which both upstairs fixed points are mapped.

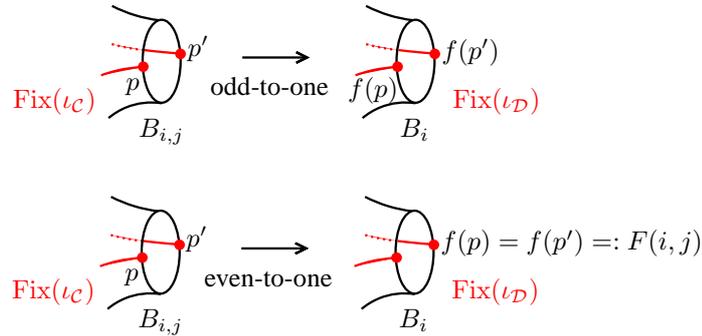
\begin{center}
	\label{fig-oddevenmap}
	\input{pics/oddmap.TpX}
	\input{pics/evenmap.TpX}
\end{center}

\begin{definition}\label{def:openhn}
Fix $g,d, \mu$. Fix $\iota$ as above with $\mu_{\iota(i,j)} = \mu_{(i,j)}$.
Let $\mathcal D$ be a surface of genus $h$ with $n$ boundary circles $B_1,\ldots,B_n$ 
and with orientation-reversing involution $\iota_{\mathcal D}$ such that 
$\iota_{\mathcal D}(B_i) = B_{\iota(i)}$. 
Finally, fix a map
	\[
	  F : \{(i,j) \in \Fix(\iota) : \mu_{i,j} \equiv 0 \; \mod 2\} \to \Fix(\iota_{\mathcal D}) \cap (B_1 \cup \ldots \cup B_n)
	\]
with $F(i,j) \in B_i$.
We define the \emph{refined real Hurwitz number} 
$H^{\refi, \R}_{g}((\mathcal D,\iota_{\mathcal D}), \mu, \iota, F)$ to be the weighted number 
of marked (unramified) real covers $f:\mathcal C \rightarrow \mathcal D$ 
of given genus and ramification profile and such that
\begin{itemize}
  \item $\iota_{\mathcal C}(B_{i,j}) = B_{\iota(i,j)}$,
	\item $f(\Fix(\iota_{\mathcal C} \cap B_{i,j})) = F(i,j)$ (if $\mu_{i,j}$ is even).
\end{itemize}
Each cover is weighted by one over the number of automorphisms (of marked real covers). 
\end{definition}

The independence statement of the representative in a homeomorphism class (respecting involutions and boundary circles) analogous to remark 
\ref{rem:independence} also holds here.

\section{Tropical curves, real structures and covers}\label{sec-tropicalcurves}

A tropical curve can be thought of as the metrization of a graph. 
More precisely, we start with a topological graph, i.e.\
a connected topological space $C$ which locally around each point $p$ is homeomorphic 
to a star with $r$ halfrays. 
The number $r$ is called the \emph{valence} of the point $p$ and denoted by $\val(p)$.
We require that there are only finitely many points with $\val(p)\neq 2$.
Additionally, we equip $C$ with a weighting, i.e.\
a function $g:C\rightarrow \N$ which is non-zero at finitely many points.
We call the number $g(p)$ the \emph{genus} at $p$. 
We require $g(p)=0$ if $val(p)=1$.
We refer to the points with $\val(p)>2$ or $g(p)>0$ as the \emph{(inner) vertices} of $C$. 
We assume that $C$ has at least one inner vertex. 
The points with $\val(p)=1$ are called \emph{leaves}.
The set of vertices and leaves is denoted by $\VV(C)$.
Finally, we equip $C\setminus\{\text{leaves}\}$ with a complete inner metric. 
$C$ together with this metric  is called a \emph{tropical curve}. 
By abuse of notation, the underlying graph with vertex set $\VV(C)$ is also denoted by $C$. 
Correspondingly, we can speak about \emph{edges} and \emph{flags} of $C$. 
The set of edges is denoted by $\EE(C)$.
We think of edges as subsets in $C$. After removing leaves, any edge is either
isometric to a closed interval $[0,a]$ or a halfray $[0,\infty)$. Edges of the 
first type are called \emph{inner edges}, and $a$ is called their \emph{length}. 
Edges of the second type are called \emph{ends}.
Note that the metric on $C$ is completely determined by the lengths for all inner edges.

 
  In a variant of this definition, we require all ends to have a finite length in $\R_\geq$ 
	just as the internal edges. The resulting object is called an \emph{open tropical curve}. 
	Of course, we could combine these two definitions and allow partially open tropical curves 
	for which some ends are infinitely long and some ends have a finite length. 
	We avoid this to keep notation simpler.
  
 An isomorphism of tropical curves is a homeomorphism respecting the metric data and the genus function.

The \emph{genus} of a tropical curve $C$ is defined to be the sum of the first Betti number of $C$ as a graph and $\sum_{v\in C} g(v)$.

\begin{definition}
  \label{def:prerealstructure}
  Let $C$ be an (open) tropical curve. A \emph{prereal structure} on $C$ is given by an automorphism $\iota : C \to C$ with $\iota^2 = \id$. A tropical curve with a prereal structure is called a prereal tropical curve.
\end{definition}

Given a prereal structure, we denote by $\VV_\fix (C)$ the set of vertices $v$ with $\iota(v) = v$ and by 
$\EE_\fix (C)$ the set of edges $e$ with $\iota|_e = \id_e$.

In the following, we enhance a prereal tropical curve with more data similar to metrized complexes (see e.g.\ \cites{AB12, ABBR13}). 
For this purpose, we denote by $S_{g,n}$ an oriented topological surface of genus $g$ and with $n$ boundary circles for a given genus $g\in \N$ and a number $n$.

\begin{definition}
  \label{def:realstructure}
	Let $C$ be a tropical curve with prereal structure $\iota$. A \emph{real structure} on $C$ is given by the 
	following additional data. 
	\begin{enumerate}
		\item 
	    For each $v \in \VV_\fix(C) $ we consider the
    	oriented topological surface $S_v := S_{g(v),\val(v)}$ and fix
    	\begin{itemize}
	    	\item 
	    	  a labelling of the boundary circles of $S_v$ by the flags adjacent to $v$
					(we denote the boundary circle corresponding to the flag $f$ by $B_f$),
	    	\item
	    	  a orientation-reversing involution $\iota_v$ on $S_v$,
    	\end{itemize}
    	such that labelling and involutions are compatible, i.e.\ $\iota_v(B_f) = B_{\iota(f)}$
			for all flags $f$. 
		\item
			For each edge $e \in \EE_\fix(C)$ given by the two flags $f, f'$ adjacent to $v, v'$,
			we consider the two fixed points of  $\iota_v|_{B_f}$ resp.\ $\iota_{v'}|_{B_{f'}}$
			and choose a identification between them (one of the two possible). We can now speak 
			of the \emph{fixed points of $e$} and denote this set (of two elements) by $F_e$.

	\end{enumerate}
	A tropical curve together with a fixed real structure $(C,\iota_C,(S_v)_{v\in \VV(C)})$ is called a \emph{real tropical curve}. 
	By abuse of notation, we sometimes denote a (pre-)real tropical just by $C$.
	\end{definition}

\begin{remark}
  \label{rem:realstrindsurf}
	The definition of a real structure on a tropical curve $C$ is tailored such that we can construct a ``global'' topological surface
	with orientation-reversing involution from it. This is done in two steps.
	\begin{enumerate}
		\item 
		  For each vertex $\VV(C) $ we take oriented copy of $S_v := S_{g(v), \val(v)}$
			and extend the labellings and real involutions to the case of non-fixed vertices, 
			i.e.\ for each $v \notin \VV_\fix(C)$, we also fix a labelling of the boundary circles
			of $S_v$ by the flags adjacent to $v$ and we pick orientation-reversing
			homeomorphisms $\iota_v : S_v \to S_{\iota(v)}$ 
			with $\iota_v \circ \iota_v  = \id_{S_v}$ and $\iota_v(B_f) = B_{\iota(f)}$ for all flags $f$. 
			Additionally, if $v $ is a leaf vertex adjacent to an infinitely long edge, we mark the disc $S_v$ by choosing a point 
			$p_v \in \text{Int}(S_v)$ such that $\iota_v(p_v) = p_{\iota(v)}$.
		\item
      For each edge $e \in \EE(C)$ given by the two flags $f, f'$ adjacent to $v, v'$,
			we glue the surfaces $S_v$ and $S_{v'}$ along the boundary circles $B_f$ and $B_{f'}$ via 
			a homeomorphism $g_f : B_f \to B_{f'}$ such that
			\begin{itemize}
				\item 
				  $g_f$ reverses orientations (in order to make the orientations of the glued surfaces compatible),
				\item
				  $g_f$ is compatible with the local real involutions, i.e.\ the diagram
			    \begin{equation}
			      \begin{CD}
              B_{f} @>{\iota_{v}}>> B_{\iota(f)}\\
              @V{g_{f}}VV @VV{g_{\iota(f)}}V\\
              B_{f'} @>>{\iota_{v'}}> B_{\iota(f')}\\
            \end{CD}
			    \end{equation}
			    commutes,
				\item
				  if $e \in \EE_\fix(C)$, then the identification of the fixed points chosen in 
					part (b) of definition \ref{def:realstructure} agrees with the one given by $g_f$
					(by the previous condition $g_f$ maps fixed points to fixed points).
			\end{itemize}
	\end{enumerate}
	Obviously, homeomorphisms $g_f$ satisfying these conditions exist. After gluing we obtain
	a oriented topological surface $\CCC$ of genus $g(C)$, with marked points labelled by the leaves of $C$ (in the case of an open tropical curve, with boundary circles labelled by the leaves of $C$). Moreover, the local real involutions
	can be glued as well and give rise to a global real involution $\iota_\CCC : \CCC \to \CCC$.
\end{remark}

A \emph{morphism} $\varphi:C\rightarrow D$ \emph{of (open) tropical curves} is a harmonic map of metric graphs satisfying the local Riemann-Hurwitz condition at each point. More precisely, $\varphi$ is surjective and piecewise integer affine linear, the slope of $\varphi$ on a flag or edge $e$ is called the \emph{weight} $\omega(e)\in \N_{>0}$. 
(The case of contracted edges is not relevant for our purposes.) The harmonicity (also referred to as balancing) states that for each point $v\in C$, the number $d_v=\sum_{f\mapsto f'} \omega(f)$ (where $f'$ is a chosen flag adjacent to $\varphi(v)$ and the sum goes over all flags $f$ mapping to $f'$) does not depend on the choice of $f'$. This number is called the \emph{local degree of $\varphi$ at $v$}. The local Riemann-Hurwitz condition states that when $v\mapsto v'$ with local degree $d$ 
$$2-2g(v)= d(2-2g(v'))-\sum (\omega(e)-1),$$ where the sum goes over all flags $e$ adjacent to $v$. 

\begin{definition}\label{def:complexlocalhn}
Let $\varphi:C\rightarrow D$ be a morphism of (open) tropical curves, and assume that both $C$ and $D$ contain only one inner vertex $v$ resp.\ $v'$. 
Then the data specified in \ref{def:complexhn} to define a (complex) Hurwitz number is encoded in the tropical cover: for each flag $f$ 
of $v'$, we can define a vector $\mu_f$ of ramification indices by collecting the weights of the flags of $v$ mapping to $f$. 
By harmonicity (resp.\ balancing), the $\mu_f$ are all partitions of the same degree, namely the local degree of $\varphi$ at $v$. 
The local Riemann-Hurwitz condition at $v$ implies that the Riemann-Hurwitz formula is satisfied for ramified covers matching the data.
We denote $H^{\C}(\varphi,v)= H^{\ma, \C}_{g(v)}(S_{v'}, \mu)$.
Let us emphasize that the covers contributing here are marked, where we
use the flags of $C$ as labels (as the entries of $\mu$ are labelled by these flags as well). \\
For a general morphism $\varphi:C\rightarrow D$ and a vertex $v$ of $C$, we can cut the edges adjacent to $v$ and also the 
edges adjacent to $\varphi(v)$ in $D$, thus producing local morphisms of open tropical curves $\varphi_v:C_v\rightarrow D_{\varphi(v)}$ 
satisfying the requirement from above (here, $C_v$ denotes the link of $C$ at $v$, and $\varphi_v$ the restriction of $\varphi$ to this link). 
We denote the corresponding \emph{local Hurwitz number} as above by $H^{\C}(\varphi,v)$.
\end{definition}

\begin{definition}
  \label{def:realcover}
  Let $\iota_C$ and $\iota_D$ be prereal structures for tropical curves $C$ resp.\ $D$. We say $\varphi:C\rightarrow D$ is
	a \emph{prereal morphism} if it satisfies $\iota_D \circ \varphi = \varphi \circ \iota_C$. \\
Let $D$ be a real tropical curve. A \emph{real tropical cover} of $D$ is a prereal curve $C$ and
	a prereal morphism $\varphi:C\rightarrow D$
	together with a map
	\[
	  F : \{e \in \EE_\fix(C) : \omega(e) \equiv 0 \; \mod 2\} \to \bigcup_{e' \in \EE_\fix(D)} F_{e'}
	\]
	such that $F(e) \in F_{\varphi(e)}$.
	We often denote the cover just by $\varphi$. An \emph{isomorphism} of real covers $\varphi, \varphi'$
	is a prereal isomorphism $\alpha : C \to C'$ such that $\varphi = \varphi' \circ \alpha$ and such that
	$F(e) = F'(\alpha(e))$ for all even edges $e$.
\end{definition}
The analogous definition is made for the case of open tropical curves.

\begin{definition}
  \label{def:specialHN}
  Fix a real tropical cover $\varphi:C\rightarrow D$ of a (open) real tropical curve $D$, and assume that both $C$ and $D$ contain only one inner vertex $v$ resp. $v'$, which are necessarily both fixed vertices for $C$ and $D$.
   Notice that the data we need to specify to obtain a refined Hurwitz number as in definition \ref{def:openhn} is encoded in the real tropical cover: The vectors of ramification indices are determined by the weights of $\varphi$, as in definition \ref{def:complexlocalhn}. 
   The local Riemann-Hurwitz condition implies that the Riemann-Hurwitz formula is satisfied for a cover matching the ramification data as in definition \ref{def:openhn}, resp.\ more precisely, for a cover where we add in punctured discs for the boundary circles. The choice of fixed points is given by the map $F$, as the notation suggests.
   Thus, we can define $H^{\R}(\varphi,v)= H^{\refi, \R}_{g(v)}((S_{v'}, \iota_{v'}), \mu, \iota, F)$.
   Let us emphasize again that the labels used to mark the covers contributing to 
	$H^{\R}(\varphi,v)$ correspond to the flags of $C$. \\
   As in definition \ref{def:complexlocalhn}, we define the \emph{local Hurwitz number} $H^{\R}(\varphi,v)$ 
   of a real tropical cover $\varphi:C\rightarrow D$ at a fixed vertex $v$ of $C$ to be the number associated 
   to the open cover that we obtain by cutting the edges adjacent to $v$ and to its image and considering the morphism 
   restricted to the link at $v$, so that we obtain a cover whose source and target both contain only one vertex, as required above. \\
   Note that the numbers $H^{\R}(\varphi,v)$ depend on the real structure we fixed for the tropical curve $D$.
   \end{definition}

\begin{definition}  \label{def:multiplicityrealcover}
 Let $\varphi :C\rightarrow D$ be a real tropical cover.  
   Set $EE := \{e \in \EE_\fix(C) : \omega_\varphi(e) \text{ even}\}$, let $\EE_{\conj}(C)$ be the set of unordered pairs $(e,e')$ of edges of $C$ satisfying $\iota_C(e)=e'$ and $\VV_{\conj}(C)$ the set of unordered pairs $(v,v')$ of vertices of $C$ satisfying $\iota_C(v)=v'$. 
	We define the \emph{multiplicity} of $\varphi$ to be $ m(\varphi) :=$
	$$   \frac{2^{|EE|}}{\Aut(\varphi)} 
		  \prod_{v \in \VV_{\fix}(C)} H^\R(\varphi,v) 
		  \prod_{(e,e') \in \EE_{\conj}(C)} \omega_\varphi(e) 
			\prod_{(v,v')\in \VV_{\conj}(C)} H^\C(\varphi,v).	$$
\end{definition}
Note that this is well-defined: since $\varphi$ is compatible with the involutions, $\omega(e)=\omega(e')$ for a tuple 
$(e,e') \in \EE_{\conj}(C)$. This also implies that $H^\C(\varphi,v)=H^\C(\varphi,v')$ for a tuple $(v,v')\in \VV_{\conj}(C)$. Note also that the multiplicity
depends on the choice of real structure for $D$, since the local Hurwitz numbers depend on $S_v$.

\begin{definition}\label{def-realtrophn}
As before, fix $g,d,\mu$. Let $(D,\iota_D, (S_v)_{v\in \VV(C)})$ be a real tropical curve
with $n$ leaves $l_1, \ldots, l_n$.
We define the \emph{real tropical Hurwitz number} $H^{\trop}_g((D,\iota_D, (S_v)_{v\in \VV(C)}),\mu)$
to be the number of marked real tropical covers $\varphi:C\rightarrow D$
of given genus, degree and ramification (i.e.\ the multiset of weights of the leaves of $C$ 
mapping to $l_i$ is equal to $\mu_i$),
counted with multiplicity as defined in \ref{def:multiplicityrealcover}. 
\end{definition}

\begin{remark}
  In complete analogy to section \ref{sec-realcovers}, we can count
	marked instead of unmarked covers. Here, a marking of tropical cover
	is a labelling $l_{i,1}, \ldots, l_{i,l(\mu_i)}$ of the leaves mapping
	to $l_i$ such that $\omega(l_{i,j}) = \mu_{i,j}$. 
	The multiplicity of a marked cover is given by the same formula as in definition
	\ref{def:multiplicityrealcover}, where $\Aut(\varphi)$ denotes the group of 
	automorphisms respecting the marking.
	Let $H^{\ma, \trop}_g((D,\iota_D, (S_v)_{v\in \VV(C)}),\mu)$ be the corresponding number.
	As before we have
	\[
	  H^{\ma, \trop}_g((D,\iota_D, (S_v)_{v\in \VV(C)}),\mu) =
		  |\Aut(\mu)| \cdot H^{\trop}_g((D,\iota_D, (S_v)_{v\in \VV(C)}),\mu).
	\]
\end{remark}

\section{The correspondence theorem}\label{sec-corres}
In this section, we state and prove the correspondence theorem declaring the equality of real Hurwitz numbers as defined in \ref{def-realhn} to their tropical counterparts defined in the
previous section, definition \ref{def-realtrophn}. 
The result can easily be generalized to open or partially open Hurwitz numbers
(i.e.\ allowing surfaces with boundary), but we avoid this here for the sake of simplicity.
\begin{theorem}
  \label{thm:equivalence}
	Let $(D,\iota_D, (S_v)_{v\in \VV(C)})$ be a real tropical curve. Let $(\DDD, \iota_\DDD, \mathcal P)$ be the
	associated topological surface with real structure and with punctures, according to remark
	\ref{rem:realstrindsurf}. Fix a genus $g \geq 0$ and for each leaf $i$ of $D$ a ramification profile $\mu_i$.
	Then
	\[
	  H^\R_g((\DDD, \iota_\DDD), \mathcal P ,\mu) =H^{\trop}_g((D,\iota_D, (S_v)_{v\in \VV(C)}),\mu).
	\]
\end{theorem}
 
Of course, by our previous remarks the same equality holds for the marked numbers.

\begin{proof}
  The main strategy of the proof is analogous to \cite{BBM10}*{Theorem 2.11} and we focus on the 
	necessary generalizations for the real case here. 
  We present the proof in three steps:
	
	\emph{1. step:}
	
	Let us start with a real ramified cover $\Phi : (\CCC, \iota_\CCC) \to (\DDD, \iota_\DDD)$.
	The first step of the proof is to construct a tropical real cover $\varphi : C \to D$ associated
	to $\Phi$. Let us start by constructing the underlying graph of $C$. Remember that $\DDD$ is
	obtained from gluing simpler surfaces according to the combinatorial structure of $D$. In particular,
	$\DDD$ contains embedded circles $D_e$ labelled by the edges of $D$. If $e$ is adjacent to a leaf $i$, the circle $D_e$ bounds
	a sphere with one marked point and $\Phi$ is ramified only above these marked points (according to $\mu_i$).
	For any $e$, $\Phi^{-1}(D_e)$ is a collection of embedded circles in $\CCC$. We call a \emph{component} of $\CCC$
	(resp.\ of $\DDD$) the closure of a connected component in the surface minus all the embedded circles.
	$C$, as a graph, consists of a vertex for each component of $\CCC$ and an edge for each embedded circle, connecting
	the two vertices whose components it is adjacent to. 
	
	By our construction, $\Phi$ maps every component resp.\ embedded circle of $\CCC$ to a unique component resp.\ embedded circle 
	of $\DDD$. This defines a map $\varphi : C \to D$ on the level of graphs. To make $C$ a fully-fledged tropical curve and 
	$\varphi$ a morphism, it suffices to determine the weights $\omega_\varphi(e)$ for all edges $e$ of $C$. Together
	with the metric structure on $D$ this determines uniquely the metric/tropical structure on $C$. The weights, in turn,
	are given by the formula
	\[
	  \omega_\varphi(e) = \deg(\Phi : D_e \to D_{\varphi(e)}),
	\]
	where the right hand side denotes the degree of the corresponding $S^1$-cover.
	
	We now add a prereal structure to $C$. From the compatibility condition $\iota_\CCC \circ \Phi = \Phi \circ \iota_\DDD$
	it follows that $\iota_\CCC$ maps components to components and embedded circles to embedded circles (as $\iota_\DDD$ 
	does the same on $\DDD$ by construction). This gives an involution $\iota_C :C \to C$ which turns $C$ into a prereal curve.
	Note that 
	\[
	  \omega_\varphi(e) = \omega_\varphi(\iota_C(e)),
	\]
	hence $\iota_C$ respects the lengths of the edges of $C$. Moreover, again by construction it is clear that $\varphi$
	is a prereal morphism satisfying the condition on the weights of the ends as prescribed by the ramification data.
	
	Finally, the last piece of information missing is the map
	\[
	  F : \{e \in \EE_\fix(C) : \omega(e) \equiv 0 \; \mod 2\} \to \bigcap_{e \in \EE_\fix(D)} F_e
	\]
	As explained before definition \ref{def:openhn}, for each $e \in \EE_\fix(C)$  
	with even weight the map $\Phi : D_e \to D_{\varphi(e)}$ is of even degree and hence
	maps both fixed points of $D_e$ to the same fixed point of $D_{\varphi(e)}$.
  But remember that we can identify $F_{\varphi(e)}$ with the two fixed points of $D_{\varphi(e)}$.
	We set $F(e)$ to be the image of the two fixed points of $D_e$.
	
	\emph{2. step:}
	
	So far, we constructed a map $\text{Trop}$ 
	from the set of all real ramified covers $\Phi$ of $\DDD$ to the set of
	all tropical real covers $\varphi$ of $D$. The next step is to analyse the fibres of this map.
	So let us fix a tropical real cover $\varphi : C \to D$ and set 
	\[
	  \AAA := \text{Trop}^{-1}(\varphi). 
	\]
	For each vertex $v \in \VV^\fix(C)$, let $\BBB(v)$ be the set of 
	covers contributing to the local real Hurwitz number $H^\R(\varphi, v)$. 
	For each pair of conjugated vertices $v,v'$ of $C$, let $\BBB(v,v')$ be the set
	of covers contributing to the local complex Hurwitz number $H^\C(\varphi, v)$. We set
	\[
		  \BBB := 
			\prod_{v \in \VV_{\fix}(D)} \BBB(v) 
			\times
			\prod_{(v,v')\in \VV_{\conj}(D)} \BBB(v,v').		
	\]
	We want to compare $\AAA$ and $\BBB$, which is easier after adding some additional structure. 
	Namely, for each pair of conjugated edges $e,e'$ of $C$
	let us fix a point $p(e) \in D_{\varphi(e)}$ (and its conjugate 
	$p(e') \in D_{\varphi(e')}$). Now let $\AAA'$ be the fibre of $\varphi$ under $\Trop$, with the additional choice of
	a point in $\Phi^{-1}(p(e))$ for each such pair
	and with the additional choice of one of the two fixed points in $\Phi^{-1}(D_{\varphi(e)})$ for each 
	even $e$ in $\EE_\fix(C)$. We call these choices a \emph{gluing data}.
\begin{center}
\label{fig-gluingdata}
\input{pics/gluingdata1.TpX}
\input{pics/gluingdata2.TpX} \\
Examples of gluing data
\end{center}	
	An isomorphism of elements in $\AAA'$ is a usual 
	isomorphism which respects the gluing data. The map from $\AAA'$ to $\AAA$ forgetting
	the gluing data is essentially $\lambda$-to-1, where
	\[
	  \lambda = 2^{|EE|} \cdot
		  \prod_{(e,e') \in \EE_{\conj}(D) } \omega_\varphi(e).
	\]
	More precisely, after dividing out isomorphisms, we have
	\[
	  |\AAA' / \text{isom.} | = \lambda \cdot |\AAA / \text{isom.}|,
	\]
	where $|.|$ denotes the weighted cardinality with weights $\frac{1}{\Aut}$.
	We can add the same gluing data to each factor of $\BBB$ and obtain the set
	$\BBB'$ of local covers with gluing data such that
	\[
	  |\BBB' / \text{isom.} | = \lambda^2 \cdot |\BBB / \text{isom.}|.
	\]
	Here, an isomorphism is given by an isomorphism in each factor. 
	
	In a third step, we now show 
	\[
	  |\BBB' / \text{isom.} | = \Aut(\varphi) \cdot |\AAA' / \text{isom.} |.
	\]
	This finishes the proof, as the above formulas imply
	\begin{equation*}
	  \begin{split}
	  H^{\trop}_g((D,\iota_D, (S_v)_{v\in \VV(C)}),\mu) 
		&= \sum_\varphi \frac{\lambda}{\Aut(\varphi)} \cdot |\BBB / \text{isom.}| \\
		&= \sum_\varphi |\AAA / \text{isom.}| \\
		&=	H^\R_g((\DDD, \iota_\DDD), \mathcal P ,\mu).
		\end{split}
	\end{equation*}
	
	\emph{3. step:}
	
	In order to show 	
	\[
	  |\BBB' / \text{isom.} | = \Aut(\varphi) \cdot |\AAA' / \text{isom.} |,
	\]
  note that the additional structure of gluing data allows us to construct a bijection between
	$\BBB'$ and $\AAA'$. 
	An element of $\BBB'$ hands us a cover 
	$\Phi_v : S_v \to S_{\varphi(v)}$ for each $v \in C$, contributing to the local refined Hurwitz numbers. (For a pair of conjugated vertices
	$v,v'$, we are only given $\Phi_v$. We then set $S_{v'} := S_v$ (with opposite orientation) and
	$\Phi_{v'} := \iota_{\varphi(v)} \circ \Phi_v$.) 
	We also have an involution on the disjoint union of all the $S_v$ such that $S_v$ is mapped
	to $S_{\iota_C(v)}$ and the compatibility with $\iota_D$ holds.
	We now glue the boundary circles of the $S_v$ according to the combinatorial structure of $C$. 
	We do this in such a way that the local covers $\Phi_v$ can be extended to a global map $\Phi$
	on the glued surface $\CCC$. 
	In order to fix such a gluing, 
	for each edge $e$ and associated pair of boundary circles $B_f$ and $B_{f'}$ we have to choose points
	$p \in B_f$ and $p' \in B_{f'}$ which are mapped to the same point in $D_{\varphi(e)} \subset \mathcal D$. 
	Additionally, if $e \in \EE_\fix(C)$, in order to extend the involutions $\iota_v$ and $\iota_{v'}$
	to the glued surface, the necessary and sufficient condition is that the gluing identifies 
	fixed points of $B_f$ with fixed points of $B_{f'}$.
	If $e$ has odd weight, the above conditions leave us with a unique choice for the gluing
	(as over each fixed point downstairs there is a unique pair of fixed points upstairs, which hence
	has to be identified).
	For an even edge $e$ in $\EE_\fix(C)$ resp. a pair of conjugated edges $e,e'$,
	we choose the unique gluing which identifies the pairs of points specified in the gluing data
  (note that in the latter case, the gluing at $e$ fixes the gluing at $e'$).
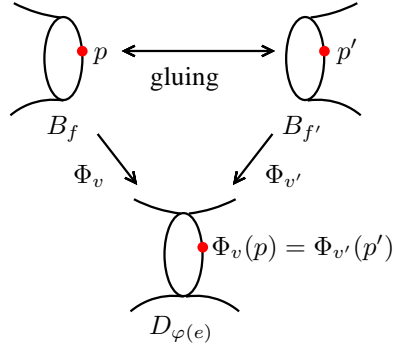
	\begin{center}
    \label{fig-gluing}
    \input{pics/gluing.TpX}
  \end{center}	
  The image of the marked points in the new circles defines a gluing data for $\Phi$.
  As our gluing data was chosen to be invariant under the involution, we can indeed extend the involution
	to $\CCC$ and hence, in total, we constructed an element in $\AAA'$.
	The inverse map is given by starting with an element $\Phi$ in $\AAA'$ and cut it into local pieces
	along the immersed circles (with gluing data induced from the gluing data on $\Phi$). \\	
	It remains to discuss isomorphisms. Let $\Phi : \CCC \to \DDD$ be an element in
	$\AAA'$. Let $\Psi : \CCC \to \CCC$ be a real automorphism of $\Phi$. 
	The way it maps components (of $\CCC$) to components and embedded circles to embedded circles
	defines a prereal automorphism $\alpha : C \to C$ of $\varphi$. This gives a surjective
	group homomorphism 
	\[
	  \Aut(\Phi) \to \Aut(\varphi).
	\]
	The kernel of this map consists of automorphisms which keep the components and
	boundary circles fixed, i.e.\ which split into local automorphisms
	on the components. This is exactly the group of automorphisms of the element in $\BBB'$
	corresponding to $\Phi$. This proves the claim.	
\end{proof}

The correspondence theorem for real Hurwitz numbers yields a graphical organization of the covers to be counted, in terms of tropical covers. 
It can also be viewed as a graphical organization of a degeneration formula for real Hurwitz numbers. 
In particular, it allows to express arbitrary real Hurwitz numbers in terms of (real) triple Hurwitz numbers 
of the sphere (i.e.\ Hurwitz numbers of the sphere with three branch points).

\section{Covers of $\P^1$ with real branch points}\label{sec-realbranch}

In this section, we focus on the case of real double Hurwitz numbers of $\P^1$ with only real branch points.
In this case, the triple Hurwitz numbers which are needed as input data for the combinatorial enumeration of covers (in terms of real tropical covers)
can be computed explicitly, allowing a completely combinatorial treatment of the counting problem.
In the case where all branch points besides $0$ and $\infty$ are real, we recover the combinatorial recipe discussed in \cite{GPMR14} using the approach
via the symmetric group. We enrich this result by a combinatorial treatment of the arbitrary case, where also negative branch points are allowed.

We need the following input data. 
We fix two natural numbers $d, g$ and two partitions of $d$, called $\lambda = (\lambda_j)$ and $\nu=(\nu_i)$. 
We set $r := 2g -2 + \ell(\lambda) + \ell(\nu)$, this is the number of additional simple branch points a generic cover of $\P^1$ with ramification profiles $\lambda$
at $0$ and $\nu$ at $\infty$ has, by the Riemann-Hurwitz formula. 

We fix $r$ points in $\R$, say $x_1 < \ldots < x_r$ and equip
each of these points with a sign $s(i) = s(x_i) \in \{+,-\}$. To this data we associate the following real tropical
curve $L$. Regardless of the real structure, the curve is just (the standard model of) $\T\P^1=\R\cup \{\pm \infty\}$ modified at the points $p_i$. Hence $L$ is
a caterpillar tree with two special ends labelled by $\pm \infty$ and $r$ ends labelled by $1, \ldots, r$.
Moreover, $w_i$ denotes the vertex obtained from modifying at $x_i$. All vertices have genus $0$. 
The real involution on $L$ is just identity, $\iota_L = \id$.
Finally, for each vertex $v_i$, let $f_{\pm \infty}$ denote the two flags pointing towards
$\pm \infty$ and let $f_i$ be the third flag (corresponding to the end $l_i$). The real structure
on $S_{v_i} = S_{0,3}$ is induced from the standard real structure on $\C\P^1$. 
Corresponding to $f_{-\infty}, f_{+\infty}$ and $f_i$, we remove open discs around $0$, $\infty$ and a third point 
$p_i \in \R \setminus \{0\} \subseteq \R\P^1$ whose sign is equal to $s(i)$.
On $B_{f_{\pm \infty}}$, call the two fixed points positive resp.\ negative, depending
on whether they touch the positive or negative part of $\R \setminus \{0\}$. 
For each edge $e \in L$, we identify the two positive and the two negative fixed points.
This gives $L$ the structure of a real tropical curve. Obviously, the real topological
surface associated to $L$ is completely determined by the numbers $n^+$ resp.\ $n^-$ of positive resp.\ negative signs $s(i)$.
Namely, the surface is obtained from $\C\P^1$ with standard real structure and punctures at $0, \infty$, $n^+$
punctures on the positive and $n^-$ punctures on the negative part of $\R\P^1 \setminus \{0,\infty\}$.
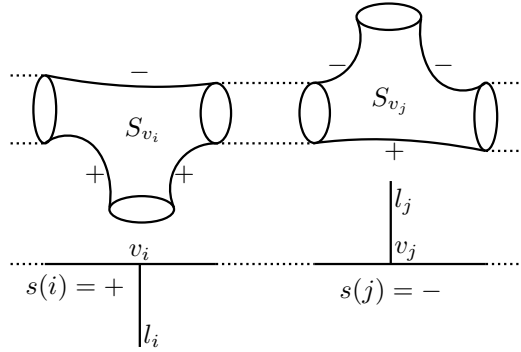
	\begin{center}
    \label{fig-curveL}
    \input{pics/curveL.TpX}
  \end{center}	
Our goal is to count real tropical covers of $L$ of genus $g$ and degree $d$ with ramification
profiles $\lambda$ resp.\ $\nu$ at $l_{-\infty}$ resp.\ $l_{+\infty}$ and
simple ramification at all other leaves. 
That is, we determine the tropical real Hurwitz numbers $H^{\trop}_g((L,\id_L, (S_v)_{v\in \VV(L)}),\mu))$, with the real structure $(S_v)_{v\in \VV(L)}$
as described above, where $$\mu=(\lambda, \nu, (2,1,\ldots,1),\ldots,(2,1,\ldots,1)).$$ 
In these situations, the count can be simplified in the following way: for each real tropical cover $\varphi:C\rightarrow L$ contributing to 
$H^{\trop}_g((L,\id_L, (S_v)_{v\in \VV(L)}),\mu))$, we contract the 
ends of $L$ adjacent to the leaves $1,\ldots,r$
and the ends of $ C$ adjacent to leaves mapping to these ends. Then we consider covers of $\T\P^1$ arising like this. Notice that covers 
arising like this do not satisfy the Riemann-Hurwitz condition at every vertex, so they are not morphisms in the sense defined above. 
Nevertheless, their properties can be described easily.

For a harmonic map $\varphi:C\rightarrow \T\P^1$, a \emph{balanced wiener} is a set of two edges of the same weight and adjacent to the same 
two vertices. The set of such pairs is denoted by $W$. A \emph{balanced fork} is a set of two ends of the same weight adjacent to the same (inner) vertex. 
The set of balanced wieners and forks is denoted by $WF$.

\begin{definition}
  \label{def:signedcover}
  A \emph{signed cover} of $\T\P^1$ is a harmonic map $\varphi : C \to \T\P^1$ with a choice of subset $I \subset WF$ of the set of balanced wieners 
  and forks (we call the elements of $I$ the \emph{conjugated} wieners resp.\ forks) and a choice of sign $S(e)\in \{+,-\}$ for every 
  internal edge $e$ of even weight which is not contained in a conjugated 
  wiener. (The set of these edges is denoted by $EE$.) 
  We require that $C$ has $r$ $3$-valent vertices $v_1, \ldots, v_r$ such that $\varphi(v_i) = x_i$. \\
  Set $I' = I \cap W$, the set of conjugated wieners.
	We associate the following \emph{multiplicity} to a signed cover.
	$$  m(\varphi) := \frac{2^{|EE|}}{2^{|WF|}} 
		  \prod_{(e,e') \in I'} \omega_\varphi(e) .$$
\end{definition}

We depict conjugated wieners/forks as bold edges, positive edges normal and negative edges dashed. Edges of odd weight
which are not in a conjugated wiener/fork are drawn normal.

\begin{definition}
  \label{def:admissiblecover}
	A signed cover $(\varphi:C\rightarrow \T\P^1, I, (S(e))_{e\in EE})$ is called \emph{real} if locally at each vertex $v_i$, the cover
	equals one of the local pictures shown in the following list, depending on the sign $s(i)$ (or their reflections, 
	with two edges coming from the left and one leaving to the right).
	If $s(i)$ is negative, the local possible pictures are:
	
	
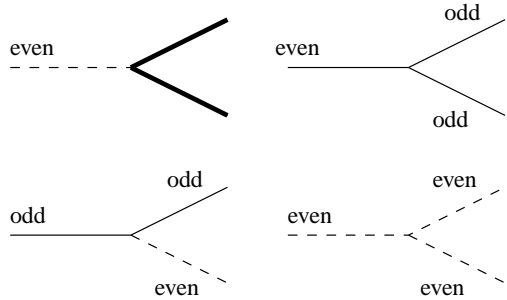
\begin{center}
\begin{picture}(0,0)%
\includegraphics{vertextypes2.pstex}%
\end{picture}%
\setlength{\unitlength}{3947sp}%
\begingroup\makeatletter\ifx\SetFigFont\undefined%
\gdef\SetFigFont#1#2#3#4#5{%
  \reset@font\fontsize{#1}{#2pt}%
  \fontfamily{#3}\fontseries{#4}\fontshape{#5}%
  \selectfont}%
\fi\endgroup%
\begin{picture}(3102,1860)(4636,-7201)
\put(6301,-5686){\makebox(0,0)[lb]{\smash{{\SetFigFont{9}{10.8}{\familydefault}{\mddefault}{\updefault}{\color[rgb]{0,0,0}even}%
}}}}
\put(4651,-6736){\makebox(0,0)[lb]{\smash{{\SetFigFont{9}{10.8}{\familydefault}{\mddefault}{\updefault}{\color[rgb]{0,0,0}odd}%
}}}}
\put(5626,-6511){\makebox(0,0)[lb]{\smash{{\SetFigFont{9}{10.8}{\familydefault}{\mddefault}{\updefault}{\color[rgb]{0,0,0}odd}%
}}}}
\put(4651,-5686){\makebox(0,0)[lb]{\smash{{\SetFigFont{9}{10.8}{\familydefault}{\mddefault}{\updefault}{\color[rgb]{0,0,0}even}%
}}}}
\put(5551,-7186){\makebox(0,0)[lb]{\smash{{\SetFigFont{9}{10.8}{\familydefault}{\mddefault}{\updefault}{\color[rgb]{0,0,0}even}%
}}}}
\put(6376,-6736){\makebox(0,0)[lb]{\smash{{\SetFigFont{9}{10.8}{\familydefault}{\mddefault}{\updefault}{\color[rgb]{0,0,0}even}%
}}}}
\put(7276,-6511){\makebox(0,0)[lb]{\smash{{\SetFigFont{9}{10.8}{\familydefault}{\mddefault}{\updefault}{\color[rgb]{0,0,0}even}%
}}}}
\put(7201,-7186){\makebox(0,0)[lb]{\smash{{\SetFigFont{9}{10.8}{\familydefault}{\mddefault}{\updefault}{\color[rgb]{0,0,0}even}%
}}}}
\put(7351,-5461){\makebox(0,0)[lb]{\smash{{\SetFigFont{9}{10.8}{\familydefault}{\mddefault}{\updefault}{\color[rgb]{0,0,0}odd}%
}}}}
\put(7276,-6136){\makebox(0,0)[lb]{\smash{{\SetFigFont{9}{10.8}{\familydefault}{\mddefault}{\updefault}{\color[rgb]{0,0,0}odd}%
}}}}
\end{picture}%

\end{center}
 If $s(i)$ is positive, the local possible pictures are:
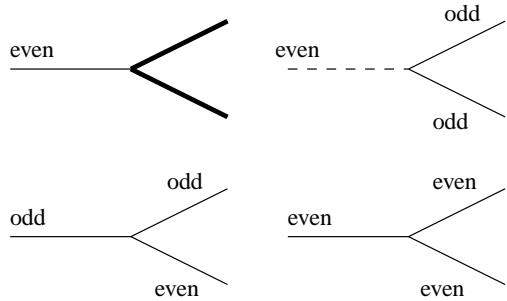
\begin{center}
\label{fig-vertextypes}
\begin{picture}(0,0)%
\includegraphics{vertextypes.pstex}%
\end{picture}%
\setlength{\unitlength}{3947sp}%
\begingroup\makeatletter\ifx\SetFigFont\undefined%
\gdef\SetFigFont#1#2#3#4#5{%
  \reset@font\fontsize{#1}{#2pt}%
  \fontfamily{#3}\fontseries{#4}\fontshape{#5}%
  \selectfont}%
\fi\endgroup%
\begin{picture}(3102,1860)(4636,-7201)
\put(6301,-5686){\makebox(0,0)[lb]{\smash{{\SetFigFont{9}{10.8}{\familydefault}{\mddefault}{\updefault}{\color[rgb]{0,0,0}even}%
}}}}
\put(4651,-6736){\makebox(0,0)[lb]{\smash{{\SetFigFont{9}{10.8}{\familydefault}{\mddefault}{\updefault}{\color[rgb]{0,0,0}odd}%
}}}}
\put(5626,-6511){\makebox(0,0)[lb]{\smash{{\SetFigFont{9}{10.8}{\familydefault}{\mddefault}{\updefault}{\color[rgb]{0,0,0}odd}%
}}}}
\put(4651,-5686){\makebox(0,0)[lb]{\smash{{\SetFigFont{9}{10.8}{\familydefault}{\mddefault}{\updefault}{\color[rgb]{0,0,0}even}%
}}}}
\put(5551,-7186){\makebox(0,0)[lb]{\smash{{\SetFigFont{9}{10.8}{\familydefault}{\mddefault}{\updefault}{\color[rgb]{0,0,0}even}%
}}}}
\put(6376,-6736){\makebox(0,0)[lb]{\smash{{\SetFigFont{9}{10.8}{\familydefault}{\mddefault}{\updefault}{\color[rgb]{0,0,0}even}%
}}}}
\put(7276,-6511){\makebox(0,0)[lb]{\smash{{\SetFigFont{9}{10.8}{\familydefault}{\mddefault}{\updefault}{\color[rgb]{0,0,0}even}%
}}}}
\put(7201,-7186){\makebox(0,0)[lb]{\smash{{\SetFigFont{9}{10.8}{\familydefault}{\mddefault}{\updefault}{\color[rgb]{0,0,0}even}%
}}}}
\put(7351,-5461){\makebox(0,0)[lb]{\smash{{\SetFigFont{9}{10.8}{\familydefault}{\mddefault}{\updefault}{\color[rgb]{0,0,0}odd}%
}}}}
\put(7276,-6136){\makebox(0,0)[lb]{\smash{{\SetFigFont{9}{10.8}{\familydefault}{\mddefault}{\updefault}{\color[rgb]{0,0,0}odd}%
}}}}
\end{picture}%

\end{center}
\end{definition}

With the following two lemmata, we classify all possible vertices of real tropical covers contributing to
$H^{\trop}_g((L,\id_L, (S_v)_{v\in \VV(L)}),\mu))$. When shrinking ends, these vertices either become $2$-valent and disappear (lemma \ref{lem:berechnung1}), or
they can be identified with 
the pictures in definition \ref{def:admissiblecover} (lemma \ref{lem:berechnung2}).

Let $T$ be an open three-valent line with standard real structure. 
That is, the involution is the identity, and the surface associated to the vertex of $T$ is $\P^1$ with three boundary circles around
punctures in the real part.

We set $g=0$ and assign the
	ramification profiles $(d)$, $(d)$ and $(1,\ldots,1)$ to the three leaves of $T$. 
	It follows from the Riemann-Hurwitz condition that any cover $\varphi : S \to T$ has one inner vertex $v$
	to which all $d+2$ ends are adjacent.
	We will now compute all non-zero local Hurwitz numbers $H^{\R}(\varphi,v)$ matching this situation.
	We use the following notation: Fixing one leaf $l$ of $T$, we denote by $\alpha(\varphi)$ 
	the number of real automorphisms of $\varphi$ which are non-trivial only on $\varphi^{-1}(l)$.
	Here, we choose $l$ to be the leaf with ramification $(1,\ldots,1)$, 
	hence $\alpha(\varphi) = |\Aut(\varphi)|$.

\begin{lemma}
  \label{lem:berechnung1}
The only non-zero local Hurwitz numbers $H^{\R}(\varphi,v)$ for a cover  $\varphi : S \to T$ of an open three-valent line $T$ with ramification profiles $(d)$, $(d)$ and 
$(1,\ldots,1)$ are given in the following list. 
(Note that the map $F$ (the choice of fixed points for each even edge) is encoded in the following pictures
by marking the points in $\text{Im}(F)$ in red.)
	\begin{itemize}
		\item If $d$ is odd then
			\begin{center}
				\input{pics/onevertex1.TpX}	
				\input{pics/pairofpants.TpX}
			\end{center}
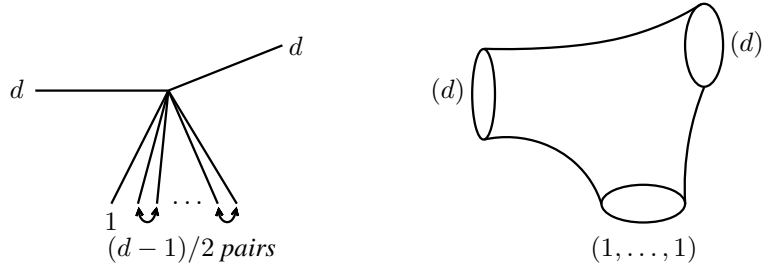
			\[ H^\R(\varphi,v) = \alpha(\varphi) = 2^{\frac{d-1}{2}} \left(\frac{d-1}{2}\right)! \]
		\item If $d$ is even: either the two fixed points for the two even fixed edges of $S$ are connected by a real arc, then
			\begin{center}
				\input{pics/onevertex2.TpX}
				\input{pics/pairofpants1.TpX}
			\end{center}
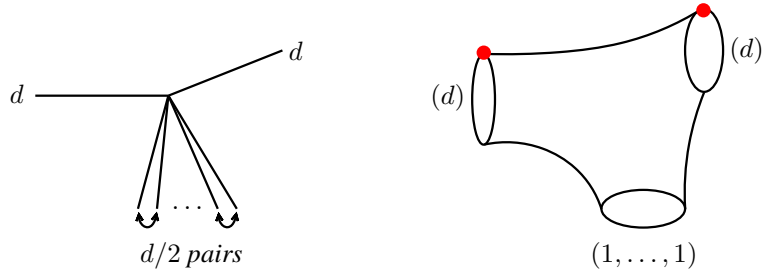
			\[  H^\R(\varphi,v) = \frac{1}{2} \cdot \alpha(\varphi) = 2^{\frac{d}{2}-1} \left(\frac{d}{2}\right)! \]
		\item or the fixed points are endpoints of the two real arcs leading to the boundary circle for which we impose ramification $(1,\ldots,1)$, then
			\begin{center}
				\input{pics/onevertex3.TpX}
				\input{pics/pairofpants2.TpX}
			\end{center}
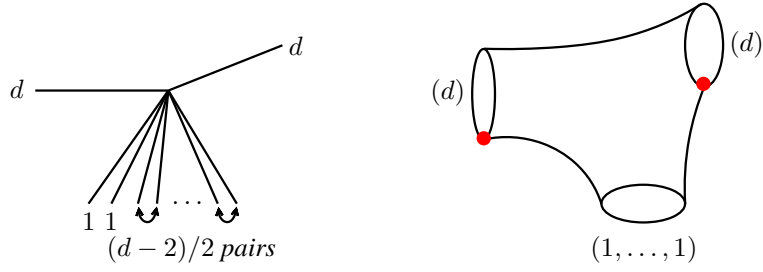
		\[ H^\R(\varphi,v) = \frac{1}{2} \cdot \alpha(\varphi) = 2^{\frac{d-2}{2}} \left(\frac{d-2}{2}\right)! \]
	\end{itemize}
\end{lemma}

\begin{remark} \label{unmarkednumbers}
  Note that the lemma just provides a sophisticated way
	of saying
	\[
    H^\R_0((\C\P^1, \conj), \{0,1,\infty\}, ((d),(d),(1,\ldots,1))) = 1,
	\]
	where in the even case two covers of weight $1/2$ contribute. Remember that
	in the complex case we have 
	$H^\C_0(\C\P^1, \{0,1,\infty\}, ((d),(d),(1,\ldots,1))) = 1/d$.
\end{remark}

\begin{proof}
  Requiring ramification profile $(1,\ldots,1)$ over one boundary circle means the map is not ramified there --- we may as well fill in the 
  corresponding discs and consider
	unramified maps between cylinders. We may identify the real structure on the target cylinder as
	\begin{equation*}
		\begin{split}
      \iota : S^1 \times [0,1] & \to S^1 \times [0,1], \\
			        (z,t)              & \mapsto (\overline{z},t).
		\end{split}
	\end{equation*}
	There is exactly one (non-real) cover of degree $d$ up to isomorphism, namely 
	\begin{equation*}
		\begin{split}
      \varphi' : S^1 \times [0,1] & \to S^1 \times [0,1], \\
			           (z,t)              & \mapsto (z^d,t).
		\end{split}
	\end{equation*}
	It has $d$ (non-real) automorphisms. 
	A compatible real structure on the source cylinder is given by
	choosing its fixed point locus of the form $\{\pm \rho\} \times [0,1]$,
	where $\rho$ is a $2d$-th root of unity. We have $d$ choices for that. 
	If $d$ is odd, these choices are killed by the $d$ automorphisms above.
	If $d$ is even, one of the automorphisms is given by $z \mapsto -z$ and hence
	keeps the fixed point locus unchanged. We are thus left with two choices, 
	corresponding to $\rho^d = \pm 1$, and each of these has a non-trivial automorphism.
	It easy to see that these cases correspond to the three real tropical covers
	shown in the statement. 
	It remains to count the number of markings in each case, i.e.\ the number
	of ways we can label the preimages of the ``trivial'' branch point by
	the flags of $S$ such that the involutions are compatible. This number
	is exactly $\alpha(\varphi)$.
\end{proof}

Now we consider again an open three-valent line $T$ with standard real structure, and set $g=0$. 
We assign the
	ramification profiles $(d)$, $(a,b)$ and $(2,1,\ldots,1)$ to the three leaves of $T$. Again, by the Riemann-Hurwitz formula,
	the only covers $\varphi:S\rightarrow T$ have one inner vertex $v$. The possibilities for local Hurwitz numbers
	matching this requirements are listed below.
As before, the surface associated to the vertex of $T$ is a sphere with three boundary circles.
The possibilities for local Hurwitz numbers in this case depend on the choice of orientation for the target
	pair of pants or, in other words, to the two possible cyclic orderings
	of the three branch points on the real locus (after fixing the orientation). 
	In lemma \ref{lem:berechnung1}, we make the choice requiring ramification $(d)$ over $\infty$, ramification $(a,b)$ over $0$ and simple
	ramification over $-1$ with $0,-1,\infty \in \C\P^1$ with standard real structure and orientation.
	The opposite choice is treated analogously.

Again, we denote by $\alpha(\varphi)$ the number of real automorphisms of 
$\varphi$ which are non-trivial only on $\varphi^{-1}(l)$,
where now $l$ is the leaf with simple ramification $(2,1,\ldots,1)$. 
In other words, $\alpha(\varphi) = |\Aut(\varphi)|$ if $a \neq b$ and 
$\alpha(\varphi) = |\Aut(\varphi)|/2$ if $a = b$.

\begin{lemma}
  \label{lem:berechnung2}
  The only non-zero local Hurwitz numbers $H^{\R}(\varphi,v)$ for a cover  $\varphi : S \to T$ of an open three-valent line $T$ with ramification profiles 
  $(d)$ (over $\infty$), $(a,b)$ (over $0$) and 
$(2,1,\ldots,1)$ (over $-1$) are given in the following list. 
Again, we encode the map $F$ by marking the points in $\text{Im}(F)$ in red
(in the third case, where we have two fixed even edges mapping to the right boundary circle, the same fixed point
is chosen, indicated by a ``$2\times$'' in the picture).
\begin{itemize}
		\item If $d$ is odd, then
			\begin{center}
				\input{pics/onevertex21.TpX}	
				\input{pics/pairofpants21.TpX}
			\end{center}
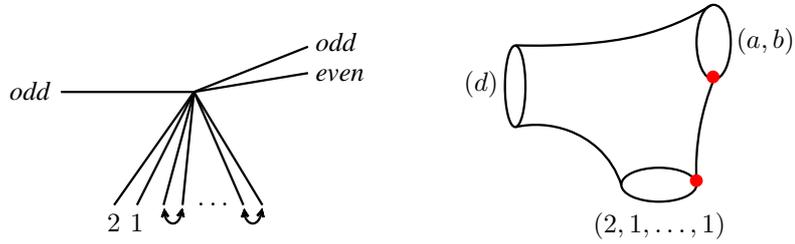
			\[  H^\R(\varphi,v) = \alpha(\varphi) = 2^{\frac{d-3}{2}} \left(\frac{d-3}{2}\right)! \]
		\item If $d$ is even, and $a,b$ odd, then
			\begin{center}
				\input{pics/onevertex22.TpX}	
				\input{pics/pairofpants22.TpX}
			\end{center}
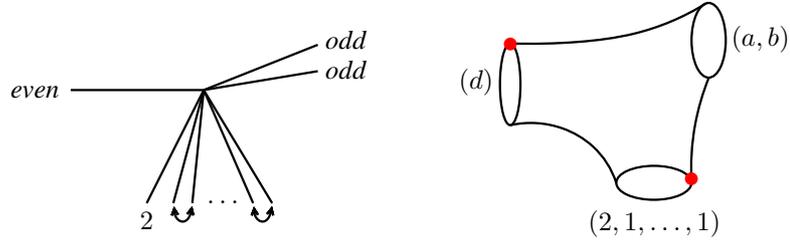
			\[ H^\R(\varphi,v) = \alpha(\varphi) = 2^{\frac{d-2}{2}} \left(\frac{d-2}{2}\right)! \]
		\item If $d$, $a$ and $b$ are even, then
			\begin{center}
				\input{pics/onevertex23.TpX}	
				\input{pics/pairofpants23.TpX}
			\end{center}
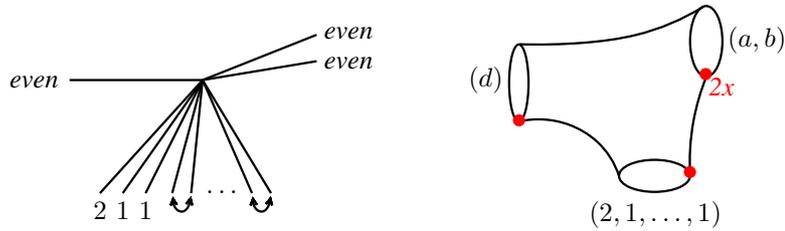
			\[  H^\R(\varphi,v) = \alpha(\varphi) = 2^{\frac{d-2}{2}} \left(\frac{d-4}{2}\right)! \]
		\item If $d$ is even and the involution on $a=b$ is non-trivial, then
			\begin{center}
				\input{pics/onevertex24.TpX}	
				\input{pics/pairofpants24.TpX}
			\end{center}
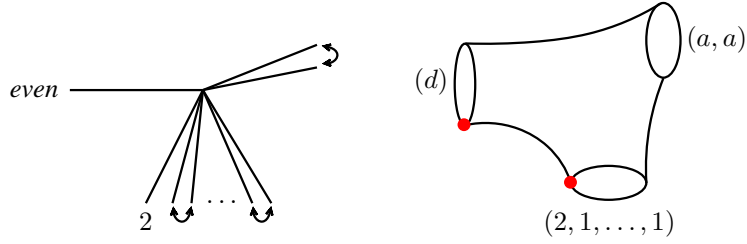
			\[  H^\R(\varphi,v) = \alpha(\varphi) = 2^{\frac{d-2}{2}} \left(\frac{d-2}{2}\right)! \]
	\end{itemize}
\end{lemma}

\begin{remark}
  Again the lemma states in a sophisticated way that 
	\[
    H^\R_0((\C\P^1, \conj), \{0,1,\infty\}, ((d),(a,b),(2,1,\ldots,1))) = 1,
	\]
	where in the even case $a=b$ two covers of weight $1/2$ contribute. 
\end{remark}

\begin{proof}
  Let us first show that the four cases in the statement are indeed the only ones 
	with non-zero Hurwitz number.
  Disregarding the real structures on the source surface, 
	there is a unique homeomorphism type of cover
	of given ramification which can be described by the map
	\begin{equation*} 
		\begin{split}
      \varphi' : \C & \to \C, \\
			           z  & \mapsto \lambda (z-1)^a (z+1)^b,
		\end{split}
	\end{equation*}
	with $\lambda = (-1)^{a-1}\frac{d^d}{2^d a^a b^b}$ (the simple ramification point is at $z_s=\frac{b-a}{d}$). 
	This cover has no non-trivial automorphisms,
	except for $a=b$, when there is exactly one (given by $z \mapsto -z$). We now have to 
	analyse the possibilities to equip the source Riemann sphere with a real structure $\iota$ compatible with
	$\varphi'$. 
	Let us focus on the fixed point sets for $\iota$. First note that $z_s$ must always be a fixed point, since
	it is the only ramification point over $-1$. Near $z_s$, the preimage $\varphi'^{-1}(\R\P^1)$
	looks like the union of two smooth arcs --- one of them the standard real line $\R\P^1$, the other
	one with tangent direction $\pm i$. 
	Notice that the cyclic order of the ramification points on the standard real line is $\infty, -1,z_s,1$.
	The fixed point set of $\iota$ must be a circle that contains one of these two
	arcs. 
	In the first case $\Fix(\iota) = \R\P^1$, we obtain the standard involution $\conj : z \mapsto \overline{z}$.
	This corresponds to the first three cases from above.
	Let us now consider the second case, i.e.\ $\Fix(\iota)$ contains the arc with tangent direction
	$\pm i$. Note that this arc, and hence $\Fix(\iota)$, is invariant under $\conj$. 
	It follows that $\Fix(\iota)$ intersects $\R\P^1$ in exactly one further point, which must 
	be $\pm 1$ or $\infty$. The remaining two ramification points must be exchanged by $\iota$ since they do not belong to the fixed circle.
	It follows from the cyclic ordering of the four ramification points on the standard real line that the point of intersection must be $\infty$
	and $\iota$ must exchange $\pm 1$.	
	
 	\begin{center}
   	\input{pics/reallocus.TpX}		
 	\end{center}		
	
	This is only possible if $a=b$. It follows that in this case $\iota$ must be equal to $z \mapsto -\overline{z}$. This 
	corresponds to the fourth case in the list of the statement.
	
	So far, we proved that the four cases in the statement are indeed the only non-zero Hurwitz number and in each
	case their is only one (unmarked) cover that contributes. It remains to calculate the number of markings
	and automorphisms. Note that when $a=b$, the choice of marking the two preimages of $0$ is cancelled by
	the (unmarked) automorphism $z \mapsto -z$. What remains is the number of markings for the preimages
	of the simple branch point, but this is exactly $\alpha(\varphi)$. 
\end{proof}

\begin{construction}\label{const:shrink}
 
 Let $\varphi:C\rightarrow L$ be a real tropical cover contributing to a Hurwitz number $$H^{\trop}_g((L,\id_L, (S_v)_{v\in \VV(L)}),\mu)),$$ 
 where $L$ is a real tropical line as
 defined at the beginning of this section.
 
 Shrink the ends adjacent to the leaves marked with $1,\ldots, r$ in $L$. Accordingly shrink the ends of $C$ mapping to these ends, producing a 
 prereal tropical curve $C'$ and a harmonic map $\varphi':C'\rightarrow \T\P^1$. We drop two-valent vertices that appear as a consequence of shrinking 
 ends and merge their adjacent edges to one edge of the appropriate length.
 Notice that the Riemann-Hurwitz condition implies that we obtain a three-valent curve $C'$ with $r$ inner vertices.
 The involution of $C$ induces an involution on $C'$, whose non-fixed locus can only consist of wieners and balanced forks. We choose the conjugated edges in
 the non-fixed locus as the set $I$ of conjugated wieners resp.\ balanced forks.
 We choose as sign for the fixed even edges of $C'$ the sign of the chosen fixed point for these edges. Notice that it follows from lemma 
 \ref{lem:berechnung1} that this is well-defined: if an even fixed edge of $C'$ originally comes subdivided by a $2$-valent vertex, then the two 
 adjacent edges of $C$ which are not shrunk have the same fixed point.
 
 In terms of our drawing conventions, this means that we 
 draw the fixed even edges of $C'$ whose chosen fixed point is drawn at the bottom of the circle as dotted lines, and those whose fixed point is drawn at the top of the circle as normal lines.
 The conjugated edges are drawn bold.
 
 \end{construction}

\begin{proposition} \label{prop-counting}
 The procedure described in construction \ref{const:shrink} is a bijection between the set of 
real tropical covers contributing to $H^{\trop}_g((L,\id_L, (S_v)_{v\in \VV(L)}),\mu))$ 
and real signed covers whose weights of ends are given by $\lambda$ and $\nu$.
Moreover, the multiplicity of a real tropical cover $\varphi:C\rightarrow L$ 
and its image $\varphi':C'\rightarrow \T\P^1$ are the same.
 %
\end{proposition}

\begin{proof}
 We have seen in \ref{const:shrink} already that we obtain a signed cover $\varphi:C'\rightarrow \T\P^1$ from $\varphi:C\rightarrow L$ by shrinking ends. 
 Applying our drawing conventions, we see that the pictures listed in lemma \ref{lem:berechnung2} indeed yield the first four pictures of definition \ref{def:admissiblecover},
 while the analogous pictures we obtain for the opposite choice of orientation of the target surface yield the second four pictures.
 So, indeed, we obtain a real signed cover by shrinking ends. Vice versa, we produce a real tropical cover of $L$ by growing ends at the images of three-valent vertices 
 of $C'$ and their preimages. It follows from lemma \ref{lem:berechnung1} and \ref{lem:berechnung2} that there is a unique way to grow ends and extend the
 involution, up to the choice of markings of the newly attached ends.
 
 It remains to prove the statement about the multiplicity.
First note that the automorphisms groups of $\varphi$ and $\varphi'$ exactly differ
by the automorphisms which exchange the shrunken leaves. To be precise, for each vertex $v$ of $C$
let $\alpha(\varphi,v)$ be the number of automorphisms of $\varphi$ which only exchanges the shrunken
leaves adjacent to $v$ (when $v \notin \Fix(\iota_C)$, we consider non-real automorphisms and
hence $\alpha(\varphi,v) = d_\varphi(v)!$). Then
\begin{equation}
\label{eq:autom}
  |\Aut(\varphi)| = |\Aut(\varphi')| \prod_{[v] \in \text{Vert}(C)/\iota_C} \alpha(\varphi,v).
\end{equation}
Now fix $[v] \in \text{Vert}(C)/\iota_C$, shrink the adjacent leaves and consider the corresponding 
change of the multiplicity.
Let us first consider the case that after shrinking we get a three-valent vertex of $C'$ (in particular,
$v \in \Fix(\iota_C)$). By lemma \ref{lem:berechnung2}, we have $H^\R(\varphi,v) = \alpha(\varphi,v)$,
which cancels one of the factors in equation \eqref{eq:autom}. 
Let us now assume after shrinking the vertex is two-valent. Then by lemma \ref{lem:berechnung2}
the local Hurwitz number equals to $\alpha(\varphi,v)$ times a factor $1$ resp.\ $1/2$ 
if the subdivided edge is a fixed edge of odd resp.\ even weight and a factor $1/\omega$
if the subdivided edge is a non-fixed edge of weight $\omega$ 
(as $H^\C_0(\varphi,v) = \omega!/\omega = (\omega-1)!$,
c.f.\ \ref{unmarkednumbers}).
Note that this factor is exactly the inverse
of the factor such an edge contributes to $m(\varphi)$, in other words, it cancels with
one of the two corresponding factors in $m(\varphi)$. By shrinking all $[v] \in \text{Vert}(C)/\iota_C$
inductively, we get
\[
  m(\varphi) = \frac{2^{|EE(C')|}}{\Aut(\varphi')} 
		  \prod_{(e,e') \in \EE_{\conj}(C')} \omega_\varphi(e).
\]
It remains to note that the  automorphisms of $\varphi'$ are generated by exchanging
the two edges of a wiener or fork. Hence $|\Aut(\varphi')| = 2^{|WF|}$ and therefore
$m(\varphi) = m(\varphi')$ from definition \ref{def:signedcover}.
\end{proof}

\begin{corollary}

With the above notations, $H^{\trop}_g((L,\id_L, (S_v)_{v\in \VV(L)}),\mu))$
equals the number of signed real covers of $\T\P^1$ whose weights match $\lambda$ and $\nu$, 
counted with the multiplicity defined in \ref{def:signedcover}.
\end{corollary}

Now recall the definition of real tropical double Hurwitz number in \cite{GPMR14}. Notice that by definition, it counts real signed covers of $\T\P^1$ for which the signs $s(i)$ 
of the branch points are all positive. The following corollary follows easily.

\begin{corollary}
 If all signs $s(i)$ for branch points are positive, the equality
$$\tilde H_g(\lambda,\nu) = 
 H^{\trop}_g((L,\id_L, (S_v)_{v\in \VV(L)}),\mu))$$
holds, where $\tilde H_g(\lambda,\nu)$ denotes the tropical real double Hurwitz number 
considered in \cite{GPMR14}.
\end{corollary}
(Note that in an earlier version of \cite{GPMR14}, we consider marked versions of tropical double Hurwitz numbers, in which case the equality above becomes $\tilde H_g(\lambda,\nu) = |\Aut(\lambda)| \cdot |\Aut(\nu)| \cdot
 H^{\trop}_g((L,\id_L, (S_v)_{v\in \VV(L)}),\mu))$.)

\begin{example}\label{ex-realbranchpoints}
 In the following example, we demonstrate that the real tropical Hurwitz numbers (and with that, real Hurwitz numbers) depend on the chosen branch points, or, more concretely,
  on the chosen number of positive and negative branch points.
  We pick $g=0$, $\lambda=(5)$ and $\nu=(3,1,1)$. First consider two positive branch points. The only cover contributing to this count is depicted below. Its multiplicity
  is $2/2=1$.
  
  
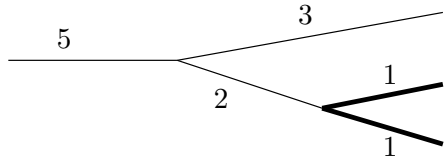
\begin{center}
   \begin{picture}(0,0)%
\includegraphics{ex5311++.pstex}%
\end{picture}%
\setlength{\unitlength}{3947sp}%
\begingroup\makeatletter\ifx\SetFigFont\undefined%
\gdef\SetFigFont#1#2#3#4#5{%
  \reset@font\fontsize{#1}{#2pt}%
  \fontfamily{#3}\fontseries{#4}\fontshape{#5}%
  \selectfont}%
\fi\endgroup%
\begin{picture}(2745,1036)(3439,-5975)
\put(5776,-5911){\makebox(0,0)[lb]{\smash{{\SetFigFont{11}{13.2}{\familydefault}{\mddefault}{\updefault}{\color[rgb]{0,0,0}$1$}%
}}}}
\put(4726,-5611){\makebox(0,0)[lb]{\smash{{\SetFigFont{11}{13.2}{\familydefault}{\mddefault}{\updefault}{\color[rgb]{0,0,0}$2$}%
}}}}
\put(5776,-5461){\makebox(0,0)[lb]{\smash{{\SetFigFont{11}{13.2}{\familydefault}{\mddefault}{\updefault}{\color[rgb]{0,0,0}$1$}%
}}}}
\put(3751,-5236){\makebox(0,0)[lb]{\smash{{\SetFigFont{11}{13.2}{\familydefault}{\mddefault}{\updefault}{\color[rgb]{0,0,0}$5$}%
}}}}
\put(5251,-5086){\makebox(0,0)[lb]{\smash{{\SetFigFont{11}{13.2}{\familydefault}{\mddefault}{\updefault}{\color[rgb]{0,0,0}$3$}%
}}}}
\end{picture}%

  \end{center}
Now assume the first branch point is negative and the second positive. Then we obtain the following two pictures with multiplicities $2/2=1$ and $2$, respectively.
As $1 \neq 3$, the Hurwitz number depends on the chosen signs for the branch points.
 \begin{center}
   \begin{picture}(0,0)%
\includegraphics{ex5311-+.pstex}%
\end{picture}%
\setlength{\unitlength}{3947sp}%
\begingroup\makeatletter\ifx\SetFigFont\undefined%
\gdef\SetFigFont#1#2#3#4#5{%
  \reset@font\fontsize{#1}{#2pt}%
  \fontfamily{#3}\fontseries{#4}\fontshape{#5}%
  \selectfont}%
\fi\endgroup%
\begin{picture}(2724,2536)(3439,-7475)
\put(5776,-5911){\makebox(0,0)[lb]{\smash{{\SetFigFont{11}{13.2}{\familydefault}{\mddefault}{\updefault}{\color[rgb]{0,0,0}$1$}%
}}}}
\put(3751,-6736){\makebox(0,0)[lb]{\smash{{\SetFigFont{11}{13.2}{\familydefault}{\mddefault}{\updefault}{\color[rgb]{0,0,0}$5$}%
}}}}
\put(5776,-7411){\makebox(0,0)[lb]{\smash{{\SetFigFont{11}{13.2}{\familydefault}{\mddefault}{\updefault}{\color[rgb]{0,0,0}$1$}%
}}}}
\put(5251,-6586){\makebox(0,0)[lb]{\smash{{\SetFigFont{11}{13.2}{\familydefault}{\mddefault}{\updefault}{\color[rgb]{0,0,0}$1$}%
}}}}
\put(4726,-7111){\makebox(0,0)[lb]{\smash{{\SetFigFont{11}{13.2}{\familydefault}{\mddefault}{\updefault}{\color[rgb]{0,0,0}$4$}%
}}}}
\put(5776,-6961){\makebox(0,0)[lb]{\smash{{\SetFigFont{11}{13.2}{\familydefault}{\mddefault}{\updefault}{\color[rgb]{0,0,0}$3$}%
}}}}
\put(4726,-5611){\makebox(0,0)[lb]{\smash{{\SetFigFont{11}{13.2}{\familydefault}{\mddefault}{\updefault}{\color[rgb]{0,0,0}$2$}%
}}}}
\put(5776,-5461){\makebox(0,0)[lb]{\smash{{\SetFigFont{11}{13.2}{\familydefault}{\mddefault}{\updefault}{\color[rgb]{0,0,0}$1$}%
}}}}
\put(3751,-5236){\makebox(0,0)[lb]{\smash{{\SetFigFont{11}{13.2}{\familydefault}{\mddefault}{\updefault}{\color[rgb]{0,0,0}$5$}%
}}}}
\put(5251,-5086){\makebox(0,0)[lb]{\smash{{\SetFigFont{11}{13.2}{\familydefault}{\mddefault}{\updefault}{\color[rgb]{0,0,0}$3$}%
}}}}
\end{picture}%

  \end{center}
\end{example}

\bibliographystyle{plain}
\bibliography{bibliographie}
\end {document}